\documentclass[a4paper,11pt]{article}

\usepackage[T1]{fontenc}
\usepackage[utf8]{inputenc} 

\usepackage[ngerman,american]{babel} 

\usepackage{amsmath}
\usepackage{amsfonts}
\usepackage{amssymb}
\usepackage{graphicx}
\usepackage{color}
\usepackage{xspace}
\usepackage{fullpage}
\usepackage{mathtools}

\usepackage{tikz}
\usepackage{subfig}
\usetikzlibrary{decorations.markings}
\usetikzlibrary{shapes.geometric}

\pgfdeclarelayer{edgelayer}
\pgfdeclarelayer{nodelayer}
\pgfsetlayers{edgelayer,nodelayer,main}
\tikzstyle{rn}=[circle,fill=white,draw=black,line width=0.6 pt]
\tikzstyle{standardedge}=[-,draw=black,line width=2.000]

\usepackage{amsthm}

\theoremstyle{plain}
\newtheorem{theorem}{Theorem}
\newtheorem{lemma}[theorem]{Lemma}
\newtheorem{fact}[theorem]{Fact}

\theoremstyle{definition}

\usepackage[colorlinks]{hyperref}
\definecolor{lightblue}{rgb}{0.5,0.5,1.0}
\definecolor{darkred}{rgb}{0.5,0,0}
\definecolor{darkgreen}{rgb}{0,0.5,0}
\definecolor{darkblue}{rgb}{0,0,0.5}

\hypersetup{colorlinks,linkcolor=darkred,filecolor=darkgreen,urlcolor=darkred,citecolor=darkblue}

\usepackage{algorithm}
\usepackage{algorithmic}

\definecolor{gray}{gray}{0.3}

\makeatletter
\providecommand*{\toclevel@algorithm}{0}
\makeatother

\usepackage{hyphenat}
\title{Minimal Asymmetric Graphs 
}

\author{Pascal Schweitzer \\ 
RWTH Aachen University\\
{\tt schweitzer@informatik.rwth-aachen.de}
\and Patrick Schweitzer\\
\\
{\tt patrick.schweitzer@gmail.com}}
\begin{document}
\maketitle

\begin{abstract}
Confirming a conjecture of Ne\v{s}et\v{r}il, we show that up to isomorphism there is only a finite number of finite minimal asymmetric undirected graphs. In fact, there are exactly 18 such graphs. We also show that these graphs are exactly the finite minimal involution-free graphs.
\end{abstract}

\section{Introduction}

A graph is \emph{asymmetric} if it does not have a nontrivial automorphism. In this paper, we are interested in asymmetric graphs that are as small as possible. An undirected graph~$G$ on at least two vertices is \emph{minimal asymmetric} if~$G$ is asymmetric and no proper induced subgraph of~$G$ on at least two vertices is asymmetric. 

In 1988 Ne\v{s}et\v{r}il conjectured at an Oberwolfach Seminar that there exists only a finite number of finite minimal asymmetric graphs, see~\cite{MR1426910}. Since then Ne\v{s}et\v{r}il and Sabidussi have made significant progress on the conjecture. They showed that there are exactly nine minimal asymmetric graphs containing~$P_5$, the path of length~$4$, as an induced subgraph \cite{MR0307955,DBLP:journals/jct/Sabidussi91,DBLP:journals/gc/NesetrilS92} and identified 18 minimal asymmetric graphs in total.
However, the conjecture has remained open over the years and has been mentioned in various other publications \cite{MR1829620, MR2249289,DBLP:journals/dm/Nesetril09, MR2920058}.
Coincidentally, Ne\v{s}et\v{r}il mentioned the open conjecture as recent as 2016 at an Oberwolfach Seminar. We now confirm the conjecture.

\begin{theorem}\label{thm:min_asym_graphs}
There are exactly 18 finite minimal asymmetric undirected graphs up to isomorphism. These are the 18 graphs depicted in Figure~\ref{fig:min_asym_graphs}.
\end{theorem}

\begin{figure}[htb]
\centering
\captionsetup[subfloat]{labelformat=empty, justification=Centering}
\subfloat[$X_1$ 
$(6,6,X_8)$\label{subfig-1}]{%\makebox[\textwidth]
	\begin{tikzpicture}
	\begin{pgfonlayer}{nodelayer}
		\node [style=rn] (0) at (-2, -0) {};
		\node [style=rn] (1) at (-2, 1) {};
		\node [style=rn] (2) at (-2.5, 2) {};
		\node [style=rn] (3) at (-3, 1) {};
		\node [style=rn] (4) at (-3, -0) {};
		\node [style=rn] (5) at (-3, -1) {};
	\end{pgfonlayer}
	\begin{pgfonlayer}{edgelayer}
		\draw [style=standardedge] (0) to (1);
		\draw [style=standardedge] (1) to (2);
		\draw [style=standardedge] (1) to (3);
		\draw [style=standardedge] (3) to (4);
		\draw [style=standardedge] (4) to (5);
		\draw [style=standardedge] (3) to (2);
	\end{pgfonlayer}
\end{tikzpicture}
}
\quad
\subfloat[$X_2$ 
$(6,7,X_7)$\label{subfig-2}]{%
	\begin{tikzpicture}
	\begin{pgfonlayer}{nodelayer}
		\node [style=rn] (0) at (-2.5, 2) {};
		\node [style=rn] (1) at (-3, 1) {};
		\node [style=rn] (2) at (-2, -0) {};
		\node [style=rn] (3) at (-3, -0) {};
		\node [style=rn] (4) at (-3, -1) {};
		\node [style=rn] (5) at (-2, 1) {};
	\end{pgfonlayer}
	\begin{pgfonlayer}{edgelayer}
		\draw [style=standardedge] (0) to (1);
		\draw [style=standardedge] (1) to (3);
		\draw [style=standardedge] (3) to (4);
		\draw [style=standardedge] (3) to (2);
		\draw [style=standardedge] (1) to (5);
		\draw [style=standardedge] (5) to (2);
		\draw [style=standardedge] (0) to (5);
	\end{pgfonlayer}
\end{tikzpicture}
}
\quad
\subfloat[$X_3$ 
$(6,7,X_6)$\label{subfig-3}]{%
	\begin{tikzpicture}
	\begin{pgfonlayer}{nodelayer}
		\node [style=rn] (0) at (-3, -3) {};
		\node [style=rn] (1) at (-2, -1) {};
		\node [style=rn] (2) at (-3, -2) {};
		\node [style=rn] (3) at (-3, -1) {};
		\node [style=rn] (4) at (-2, -3) {};
		\node [style=rn] (5) at (-2, -2) {};
	\end{pgfonlayer}
	\begin{pgfonlayer}{edgelayer}
		\draw [style=standardedge] (2) to (3);
		\draw [style=standardedge] (5) to (1);
		\draw [style=standardedge] (3) to (1);
		\draw [style=standardedge] (2) to (5);
		\draw [style=standardedge] (5) to (0);
		\draw [style=standardedge] (0) to (2);
		\draw [style=standardedge] (4) to (5);
	\end{pgfonlayer}
\end{tikzpicture}
}
\quad
\subfloat[$X_4$ 
$(6,7,X_5)$\label{subfig-4}]{%
	\begin{tikzpicture}
	\begin{pgfonlayer}{nodelayer}
		\node [style=rn] (0) at (-1, 1) {};
		\node [style=rn] (1) at (-2, -1) {};
		\node [style=rn] (2) at (-1, -0) {};
		\node [style=rn] (3) at (-1, -1) {};
		\node [style=rn] (4) at (-2, 1) {};
		\node [style=rn] (5) at (-2, -0) {};
	\end{pgfonlayer}
	\begin{pgfonlayer}{edgelayer}
		\draw [style=standardedge] (2) to (3);
		\draw [style=standardedge] (5) to (1);
		\draw [style=standardedge] (2) to (5);
		\draw [style=standardedge] (5) to (0);
		\draw [style=standardedge] (0) to (2);
		\draw [style=standardedge] (4) to (5);
		\draw [style=standardedge] (0) to (4);
	\end{pgfonlayer}
\end{tikzpicture}
}
\quad
\subfloat[$X_5$
$(6,8,X_4)$\label{subfig-5}]{%
	\begin{tikzpicture}
	\begin{pgfonlayer}{nodelayer}
		\node [style=rn] (0) at (-3, -0) {};
		\node [style=rn] (1) at (-3, -1) {};
		\node [style=rn] (2) at (-3, 1) {};
		\node [style=rn] (3) at (-2.5, 2) {};
		\node [style=rn] (4) at (-2, -0) {};
		\node [style=rn] (5) at (-2, 1) {};
	\end{pgfonlayer}
	\begin{pgfonlayer}{edgelayer}
		\draw [style=standardedge] (2) to (3);
		\draw [style=standardedge] (2) to (5);
		\draw [style=standardedge] (5) to (0);
		\draw [style=standardedge] (0) to (2);
		\draw [style=standardedge] (4) to (5);
		\draw [style=standardedge] (0) to (4);
		\draw [style=standardedge] (5) to (3);
		\draw [style=standardedge] (0) to (1);
	\end{pgfonlayer}
\end{tikzpicture}
}
\quad
\subfloat[$X_6$
$(6,8,X_3)$\label{subfig-6}]{%
	\begin{tikzpicture}
	\begin{pgfonlayer}{nodelayer}
		\node [style=rn] (0) at (-3, 1) {};
		\node [style=rn] (1) at (-2, 1) {};
		\node [style=rn] (2) at (-3, -0) {};
		\node [style=rn] (3) at (-2, -0) {};
		\node [style=rn] (4) at (-2.5, 2) {};
		\node [style=rn] (5) at (-2, -1) {};
	\end{pgfonlayer}
	\begin{pgfonlayer}{edgelayer}
		\draw [style=standardedge] (0) to (2);
		\draw [style=standardedge] (2) to (3);
		\draw [style=standardedge] (2) to (1);
		\draw [style=standardedge] (0) to (4);
		\draw [style=standardedge] (4) to (1);
		\draw [style=standardedge] (5) to (3);
		\draw [style=standardedge] (0) to (1);
		\draw [style=standardedge] (1) to (3);
	\end{pgfonlayer}
\end{tikzpicture} 
}
\quad
\subfloat[$X_7$
$(6,8,X_2)$\label{subfig-7}]{%
	\begin{tikzpicture}
	\begin{pgfonlayer}{nodelayer}
		\node [style=rn] (0) at (-3, 1) {};
		\node [style=rn] (1) at (-2, 2) {};
		\node [style=rn] (2) at (-3, 2) {};
		\node [style=rn] (3) at (-3, 3) {};
		\node [style=rn] (4) at (-2, 1) {};
		\node [style=rn] (5) at (-2, 3) {};
	\end{pgfonlayer}
	\begin{pgfonlayer}{edgelayer}
		\draw [style=standardedge] (0) to (2);
		\draw [style=standardedge] (2) to (3);
		\draw [style=standardedge] (2) to (1);
		\draw [style=standardedge] (0) to (4);
		\draw [style=standardedge] (4) to (1);
		\draw [style=standardedge] (1) to (5);
		\draw [style=standardedge] (5) to (3);
		\draw [style=standardedge] (2) to (5);
	\end{pgfonlayer}
\end{tikzpicture} 
}
\quad
\subfloat[$X_8$
$(6,9,X_1)$\label{subfig-8}]{%
	\begin{tikzpicture}
	\begin{pgfonlayer}{nodelayer}
		\node [style=rn] (0) at (-4, 2) {};
		\node [style=rn] (1) at (-3.5, 3) {};
		\node [style=rn] (2) at (-3, 2) {};
		\node [style=rn] (3) at (-3.5, 1) {};
		\node [style=rn] (4) at (-4, -0) {};
		\node [style=rn] (5) at (-3, -0) {};
	\end{pgfonlayer}
	\begin{pgfonlayer}{edgelayer}
		\draw [style=standardedge] (2) to (5);
		\draw [style=standardedge] (0) to (2);
		\draw [style=standardedge] (4) to (5);
		\draw [style=standardedge] (0) to (4);
		\draw [style=standardedge] (5) to (3);
		\draw [style=standardedge] (0) to (1);
		\draw [style=standardedge] (1) to (2);
		\draw [style=standardedge] (0) to (3);
		\draw [style=standardedge] (3) to (2);
	\end{pgfonlayer}
\end{tikzpicture}
}

\subfloat[$X_9$
$(7,6,X_{14})$\label{subfig-9}]{%
	\begin{tikzpicture}
	\begin{pgfonlayer}{nodelayer}
		\node [style=rn] (0) at (-2, 2) {};
		\node [style=rn] (1) at (-3, 2) {};
		\node [style=rn] (2) at (-3, 1) {};
		\node [style=rn] (3) at (-2, 1) {};
		\node [style=rn] (4) at (-3, -0) {};
		\node [style=rn] (5) at (-3, -1) {};
		\node [style=rn] (6) at (-2, -1) {};
		\node [style=rn] (7) at (-3, 2) {};
	\end{pgfonlayer}
	\begin{pgfonlayer}{edgelayer}
		\draw [style=standardedge] (0) to (1);
		\draw [style=standardedge] (1) to (2);
		\draw [style=standardedge] (2) to (3);
		\draw [style=standardedge] (2) to (4);
		\draw [style=standardedge] (4) to (5);
		\draw [style=standardedge] (5) to (6);
	\end{pgfonlayer}
\end{tikzpicture}
}
\quad\quad
\subfloat[$X_{10}$
$(7,7,X_{13})$\label{subfig-10}]{%
	\begin{tikzpicture}
	\begin{pgfonlayer}{nodelayer}
		\node [style=rn] (0) at (-3, 2) {};
		\node [style=rn] (1) at (-3, 1) {};
		\node [style=rn] (2) at (-2, -0) {};
		\node [style=rn] (3) at (-3, -0) {};
		\node [style=rn] (4) at (-3, -1) {};
		\node [style=rn] (5) at (-2, -1) {};
		\node [style=rn] (6) at (-3, 2) {};
		\node [style=rn] (7) at (-2, 1) {};
	\end{pgfonlayer}
	\begin{pgfonlayer}{edgelayer}
		\draw [style=standardedge] (0) to (1);
		\draw [style=standardedge] (1) to (3);
		\draw [style=standardedge] (3) to (4);
		\draw [style=standardedge] (4) to (5);
		\draw [style=standardedge] (3) to (2);
		\draw [style=standardedge] (1) to (7);
		\draw [style=standardedge] (7) to (2);
	\end{pgfonlayer}
\end{tikzpicture} 
}
\quad\quad
\subfloat[$X_{11}$
$(7,8,X_{12})$\label{subfig-11}]{%
	\begin{tikzpicture}
	\begin{pgfonlayer}{nodelayer}
		\node [style=rn] (0) at (-3, 3) {};
		\node [style=rn] (1) at (-3, 2) {};
		\node [style=rn] (2) at (-3, 1) {};
		\node [style=rn] (3) at (-2, 1) {};
		\node [style=rn] (4) at (-3, -0) {};
		\node [style=rn] (5) at (-3, 2) {};
		\node [style=rn] (6) at (-2, 2) {};
		\node [style=rn] (7) at (-2, 3) {};
	\end{pgfonlayer}
	\begin{pgfonlayer}{edgelayer}
		\draw [style=standardedge] (0) to (1);
		\draw [style=standardedge] (1) to (2);
		\draw [style=standardedge] (2) to (3);
		\draw [style=standardedge] (2) to (4);
		\draw [style=standardedge] (0) to (7);
		\draw [style=standardedge] (7) to (6);
		\draw [style=standardedge] (6) to (3);
		\draw [style=standardedge] (1) to (6);
	\end{pgfonlayer}
\end{tikzpicture} 
}
\quad\quad
\subfloat[$X_{12}$
$(7,13,X_{11})$\label{subfig-12}]{%
	\begin{tikzpicture}
	\begin{pgfonlayer}{nodelayer}
		\node [style=rn] (0) at (-5, 2) {};
		\node [style=rn] (1) at (-4.25, 0.5) {};
		\node [style=rn] (2) at (-3.5, 2) {};
		\node [style=rn] (3) at (-5, -1) {};
		\node [style=rn] (4) at (-4.25, 1.5) {};
		\node [style=rn] (5) at (-3.5, -1) {};
		\node [style=rn] (6) at (-4.25, -0.5) {};
	\end{pgfonlayer}
	\begin{pgfonlayer}{edgelayer}
		\draw [style=standardedge] (4) to (1);
		\draw [style=standardedge] (0) to (2);
		\draw [style=standardedge] (1) to (3);
		\draw [style=standardedge] (6) to (1);
		\draw [style=standardedge] (4) to (2);
		\draw [style=standardedge] (0) to (3);
		\draw [style=standardedge] (3) to (6);
		\draw [style=standardedge] (6) to (5);
		\draw [style=standardedge] (1) to (5);
		\draw [style=standardedge] (2) to (5);
		\draw [style=standardedge] (5) to (3);
		\draw [style=standardedge] (2) to (1);
		\draw [style=standardedge] (0) to (4);
	\end{pgfonlayer}
\end{tikzpicture} 
}
\quad\quad
\subfloat[$X_{13}$
$(7,14,X_{10})$\label{subfig-13}]{%
	\begin{tikzpicture}
	\begin{pgfonlayer}{nodelayer}
		\node [style=rn] (0) at (-3, -0) {};
		\node [style=rn] (1) at (-2.25, -1.5) {};
		\node [style=rn] (2) at (-2.25, -0.5) {};
		\node [style=rn] (3) at (-1.5, -3) {};
		\node [style=rn] (4) at (-1.5, -0) {};
		\node [style=rn] (5) at (-2.25, -2.5) {};
		\node [style=rn] (6) at (-3, -3) {};
	\end{pgfonlayer}
	\begin{pgfonlayer}{edgelayer}
		\draw [style=standardedge] (6) to (0);
		\draw [style=standardedge] (4) to (1);
		\draw [style=standardedge] (0) to (2);
		\draw [style=standardedge] (1) to (3);
		\draw [style=standardedge] (6) to (1);
		\draw [style=standardedge] (4) to (2);
		\draw [style=standardedge] (3) to (6);
		\draw [style=standardedge] (6) to (5);
		\draw [style=standardedge] (1) to (5);
		\draw [style=standardedge] (2) to (1);
		\draw [style=standardedge] (0) to (4);
		\draw [style=standardedge] (4) to (3);
		\draw [style=standardedge] (5) to (3);
		\draw [style=standardedge] (6) to (2);
	\end{pgfonlayer}
\end{tikzpicture} 
}
\quad\quad
\subfloat[\mbox{}\quad$X_{14}$\quad\mbox{}
$(7,15,X_{9})$\label{subfig-14}]{%
	\begin{tikzpicture}
	\begin{pgfonlayer}{nodelayer}
		\node [style=rn] (0) at (-2.25, -0.5) {};
		\node [style=rn] (1) at (-3, -3) {};
		\node [style=rn] (2) at (-1.5, -0) {};
		\node [style=rn] (3) at (-3, 0) {};
		\node [style=rn] (4) at (-1.5, -3) {};
		\node [style=rn] (5) at (-2.25, -2.5) {};
		\node [style=rn] (6) at (-2.25, -1.5) {};
	\end{pgfonlayer}
	\begin{pgfonlayer}{edgelayer}
		\draw [style=standardedge] (6) to (0);
		\draw [style=standardedge] (4) to (1);
		\draw [style=standardedge] (0) to (2);
		\draw [style=standardedge] (1) to (3);
		\draw [style=standardedge] (6) to (1);
		\draw [style=standardedge] (4) to (2);
		\draw [style=standardedge] (0) to (3);
		\draw [style=standardedge] (3) to (6);
		\draw [style=standardedge] (6) to (5);
		\draw [style=standardedge] (4) to (5);
		\draw [style=standardedge] (1) to (5);
		\draw [style=standardedge] (2) to (5);
		\draw [style=standardedge] (5) to (3);
		\draw [style=standardedge] (6) to (4);
		\draw [style=standardedge] (3) to (2);
	\end{pgfonlayer}
\end{tikzpicture} 
}

\subfloat[$X_{15}$
$(8,9,X_{18})$\label{subfig-15}]{%
	\begin{tikzpicture}
	\begin{pgfonlayer}{nodelayer}
		\node [style=rn] (0) at (-3, 1) {};
		\node [style=rn] (1) at (-3, -1) {};
		\node [style=rn] (2) at (-4, 1) {};
		\node [style=rn] (3) at (-4, -1) {};
		\node [style=rn] (4) at (-3, -0) {};
		\node [style=rn] (5) at (-4, -2) {};
		\node [style=rn] (6) at (-3.5, -0.5) {};
		\node [style=rn] (7) at (-3, -2) {};
	\end{pgfonlayer}
	\begin{pgfonlayer}{edgelayer}
		\draw [style=standardedge] (0) to (2);
		\draw [style=standardedge] (2) to (3);
		\draw [style=standardedge] (0) to (4);
		\draw [style=standardedge] (4) to (1);
		\draw [style=standardedge] (5) to (3);
		\draw [style=standardedge] (1) to (3);
		\draw [style=standardedge] (4) to (6);
		\draw [style=standardedge] (6) to (3);
		\draw [style=standardedge] (1) to (7);
	\end{pgfonlayer}
\end{tikzpicture} 
}
\quad\quad
\subfloat[$X_{16}$
$(8,10,X_{17})$\label{subfig-16}]{%
	\begin{tikzpicture}
	\begin{pgfonlayer}{nodelayer}
		\node [style=rn] (0) at (-2, -0) {};
		\node [style=rn] (1) at (-2.5, -0.5) {};
		\node [style=rn] (2) at (-2, -1) {};
		\node [style=rn] (3) at (-3, -1) {};
		\node [style=rn] (4) at (-2, 1) {};
		\node [style=rn] (5) at (-3, -2) {};
		\node [style=rn] (6) at (-3, 1) {};
		\node [style=rn] (7) at (-2.5, 0.5) {};
	\end{pgfonlayer}
	\begin{pgfonlayer}{edgelayer}
		\draw [style=standardedge] (0) to (2);
		\draw [style=standardedge] (2) to (3);
		\draw [style=standardedge] (0) to (4);
		\draw [style=standardedge] (5) to (3);
		\draw [style=standardedge] (4) to (6);
		\draw [style=standardedge] (6) to (3);
		\draw [style=standardedge] (6) to (7);
		\draw [style=standardedge] (7) to (1);
		\draw [style=standardedge] (3) to (1);
		\draw [style=standardedge] (1) to (0);
	\end{pgfonlayer}
\end{tikzpicture} 
}
\quad\quad
\subfloat[\mbox{}\hspace{1cm}$X_{17}$\hspace{1cm}\mbox{}
$(8,18,X_{16})$\label{subfig-17}]{%
	\begin{tikzpicture}
	\begin{pgfonlayer}{nodelayer}
		\node [style=rn] (0) at (-4, 4) {};
		\node [style=rn] (1) at (-3.5, 2.5) {};
		\node [style=rn] (2) at (-4, 1) {};
		\node [style=rn] (3) at (-1, 4) {};
		\node [style=rn] (4) at (-2.5, 2.5) {};
		\node [style=rn] (5) at (-2.5, 4) {};
		\node [style=rn] (6) at (-1, 1) {};
		\node [style=rn] (7) at (-1.5, 2.5) {};
	\end{pgfonlayer}
	\begin{pgfonlayer}{edgelayer}
		\draw [style=standardedge] (4) to (1);
		\draw [style=standardedge] (0) to (2);
		\draw [style=standardedge] (6) to (1);
		\draw [style=standardedge] (4) to (2);
		\draw [style=standardedge] (2) to (1);
		\draw [style=standardedge] (0) to (4);
		\draw [style=standardedge] (4) to (3);
		\draw [style=standardedge] (5) to (3);
		\draw [style=standardedge] (6) to (2);
		\draw [style=standardedge] (5) to (0);
		\draw [style=standardedge] (0) to (1);
		\draw [style=standardedge] (0) to (7);
		\draw [style=standardedge] (2) to (7);
		\draw [style=standardedge] (5) to (7);
		\draw [style=standardedge] (3) to (6);
		\draw [style=standardedge] (4) to (6);
		\draw [style=standardedge] (3) to (7);
		\draw [style=standardedge] (4) to (7);
	\end{pgfonlayer}
\end{tikzpicture} 
}
\quad\quad
\subfloat[\mbox{}\hspace{1cm}$X_{18}$\hspace{1cm}\mbox{}
$(8,19,X_{15})$\label{subfig-18}]{%
	\begin{tikzpicture}
	\begin{pgfonlayer}{nodelayer}
		\node [style=rn] (0) at (-4.5, 2.5) {};
		\node [style=rn] (1) at (-3.5, 2.5) {};
		\node [style=rn] (2) at (-6, 1) {};
		\node [style=rn] (3) at (-3, 4) {};
		\node [style=rn] (4) at (-3, 1) {};
		\node [style=rn] (5) at (-4.5, 4) {};
		\node [style=rn] (6) at (-6, 4) {};
		\node [style=rn] (7) at (-5.5, 2.5) {};
	\end{pgfonlayer}
	\begin{pgfonlayer}{edgelayer}
		\draw [style=standardedge] (6) to (0);
		\draw [style=standardedge] (4) to (1);
		\draw [style=standardedge] (0) to (2);
		\draw [style=standardedge] (1) to (3);
		\draw [style=standardedge] (6) to (1);
		\draw [style=standardedge] (4) to (2);
		\draw [style=standardedge] (6) to (5);
		\draw [style=standardedge] (1) to (5);
		\draw [style=standardedge] (2) to (1);
		\draw [style=standardedge] (0) to (4);
		\draw [style=standardedge] (4) to (3);
		\draw [style=standardedge] (5) to (3);
		\draw [style=standardedge] (6) to (2);
		\draw [style=standardedge] (5) to (0);
		\draw [style=standardedge] (0) to (1);
		\draw [style=standardedge] (0) to (7);
		\draw [style=standardedge] (2) to (7);
		\draw [style=standardedge] (6) to (7);
		\draw [style=standardedge] (5) to (7);
	\end{pgfonlayer}
\end{tikzpicture} 
}

\caption{The 18 minimal asymmetric graphs. These are also the minimal involution-free graphs. For each graph the triple~$(n,m,\text{co-}{G})$, describes the number of vertices, edges and the name of the complement graph, respectively. The graphs are ordered first by number of vertices and second by number of edges.}

\label{fig:min_asym_graphs}

\end{figure}
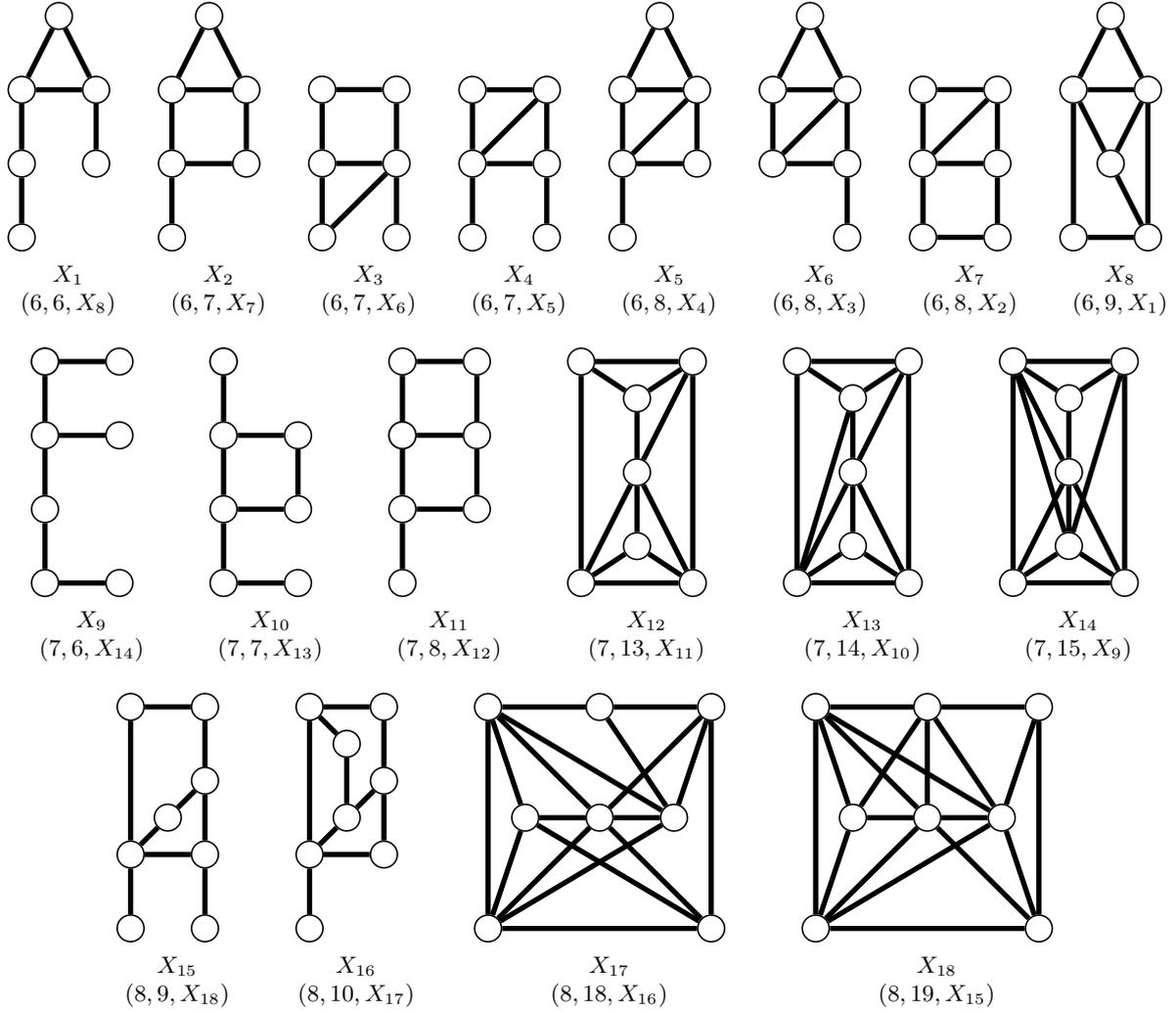

A classic result of Erd\H{o}s and R\'{e}nyi~\cite{ErdRen} says that as~$n$ tends to infinity most graphs on~$n$ vertices are asymmetric, and it is not difficult to see that every finite graph embeds into a finite asymmetric graph, so Theorem~\ref{thm:min_asym_graphs} may come as surprise.

In their papers, Ne\v{s}et\v{r}il and Sabidussi found a close connection between minimal asymmetric graphs and minimal involution-free graphs. A graph is \emph{involution-free} if it does not have an automorphism of order~2. An undirected graph~$G$ on at least 2 vertices is \emph{minimal involution-free} if~$G$ is involution-free and no proper induced subgraph of~$G$ on at least 2 vertices is involution-free.
Ne\v{s}et\v{r}il and Sabidussi conjectured that the set of finite minimal asymmetric undirected graphs and set of finite minimal involution-free undirected graphs are the same~\cite{DBLP:journals/gc/NesetrilS92}. We also confirm this conjecture, thereby determining all finite minimal involution-free undirected graphs.

\begin{theorem}\label{thm:min:inv:free}
There are exactly 18 finite minimal involution-free undirected graphs up to isomorphism. These are the 18 graphs depicted in Figure~\ref{fig:min_asym_graphs}. 
\end{theorem}
The theorem allows making more specialized statements if one is interested in particular graph classes. For example, one can observe that the only bipartite finite minimal involution-free graphs are~$X_9$,~$X_{10}$ and~$X_{11}$.
Thus every finite~$(X_9, X_{10},X_{11})$-free bipartite graph has an involution.
Similarly every finite~$(X_3, X_4)$-free split graph has an involution.

\paragraph{Notation.} All graphs in this paper are finite and simple (i.e., loopless and undirected). For a graph~$G$ we denote by~$V(G)$ and~$E(G)$ the set of vertices and edges, respectively. The \emph{open neighborhood}~$N_G(v)$ of a vertex~$v\in V(G)$ is the set of vertices adjacent to~$v$. The \emph{closed neighborhood}~$\overline{N}_G(v)$ is defined as~$N_G(v)\cup \{v\}$. For a subset of the vertices~$M\subseteq V(G)$ we denote by~$G[M]$ the subgraph of~$G$ \emph{induced} by~$M$. By~``$G$ \emph{contains} a graph~$H$'', we mean that~$G$ has an induced subgraph that is isomorphic to~$H$. 

\section{Involutions vs. automorphisms}

In this section we discuss the difference between minimal involution-free and minimal asymmetric graphs. Of course not every involution-free graph is asymmetric. 
This suggests that in principle there could be minimal involution-free graphs that are not minimal asymmetric. 
Indeed, if we consider automorphisms of order 3 instead of involutions (i.e., automorphism of order 2), it is not difficult to find minimal asymmetric graphs that are not minimal ``order-3-automorphism''-free (all graphs in Figure~\ref{fig:min_asym_graphs}) and vice versa (graphs on two vertices). 
However, when considering involutions, neither of these possibilities occurs. In fact up to isomorphism, minimal asymmetric undirected graphs and minimal involution-free undirected graphs are exactly the same. To show this, it suffices to determine all minimal involution-free graphs, observe that they all are asymmetric and use the following lemma.

\begin{lemma}\label{lem:min:inv:free:impl:asym}
If every minimal involution-free graph is asymmetric, then the minimal asymmetric graphs and the minimal involution-free graphs are exactly the same.
\end{lemma}
\begin{proof}
Let~$G$ be a minimal involution-free graph. If~$G$ is asymmetric, then~$G$ is minimal asymmetric, since every proper subgraph of~$G$ on at least two vertices has an involution and is thus not asymmetric. Under the assumption of the lemma we thus conclude that the minimal involution-free graphs are minimal asymmetric.

Suppose now that not every minimal asymmetric graph is minimal involution-free. Let~$G$ be minimal asymmetric graph that is not minimal involution-free. 
Since~$G$ is asymmetric and thus involution-free, there is a proper subgraph~$H$ of~$G$ on at least two vertices that is involution-free but not asymmetric. Choosing~$H$ as small as possible we conclude that~$H$ is minimal involution-free but not asymmetric, contradicting the assumption of the lemma.
\end{proof}

Since we show that all minimal involution-free graphs are indeed asymmetric, we are only concerned with minimal involution-free graphs throughout the rest of the paper.
We remark, however, that 
all presented proofs are applicable almost verbatim when replacing ``involution'' by ``automorphism'' and ``involution-free'' by ``asymmetric''. 

\section{An extension lemma for subgraphs of prime graphs}

A \emph{homogeneous set} of a graph~$G$ is a subset of its vertices~$M\subseteq V(G)$ with~$2\leq |M|< |V|$ such that for all~$x,x'\in M$ we have~$N_G(x)\setminus M = N_G(x')\setminus M$. 
A graph is called~\emph{prime} if it does not have homogeneous sets.
Note that for a homogeneous set~$M$, every involution of~$G[M]$ extends to an involution of~$G$ that fixes all vertices outside of~$M$. In particular this means that no subset~$M$ of a minimal involution-free graph can be a homogeneous set, since~$G[M]$ contains an involution that could be extended to an involution of the entire graph. We conclude that minimal involution-free %asymmetric
graphs are prime.  
In particular, a minimal involution-free 
undirected graph that has more than one vertex neither contains a \emph{universal vertex} (i.e., a vertex adjacent to all other vertices) nor an isolated vertex.

Two distinct vertices~$x$ and~$x'$ of a graph~$G$ are called \emph{true twins} if they have exactly the same open neighborhood, i.e.,~$N_G(x)=N_G(x')$. Note that in this case~$x$ and~$x'$ cannot be adjacent. Furthermore,~$x$ and~$x'$ are called \emph{false twins} if they have exactly the same closed neighborhood, i.e.,~$\overline{N}_G(x)=\overline{N}_G(x')$. Note that in this case~$x$ and~$x'$ must be adjacent. In either case it holds that~$N_G(x) \setminus\{x'\} = N_G(x') \setminus\{x\}$. Thus, if~$x$ and~$x'$ are true or false twins,~$\{x,x'\}$ is a homogeneous set.

We repeatedly use the following technique to extend a graph that has true or false twins.

\begin{lemma}\label{lem:the:extension:lemma}
Let~$G$ be a prime graph that contains a graph~$H$ on at least 3 vertices as an induced subgraph.
Let~$x,x' \in V(H)$ be true or false twins in~$H$. Then~$G$ contains an induced subgraph~$H'$ isomorphic to~$H$ via an isomorphism~$\varphi$ and a vertex~$v\in V(G)\setminus V(H')$ adjacent to exactly one vertex of~$\{\varphi(x),\varphi(x')\}$ such that~$\{\varphi(x),\varphi(x'),v\}$ is not a homogeneous set of~$G[V(H')\cup \{v\}]$. 
%\patrick{Let's call this copy~$K$}
\end{lemma}

\begin{proof}
Let~$H$ be an induced subgraph of~$G$ and let~$x,x'$ be true or false twins in~$H$. By possibly complementing both~$G$ and~$H$ we can assume that~$x$ and~$x'$ are adjacent.
Let~$M$ be the set of vertices~$u$ in~$(V(G)\setminus V(H)) \cup \{x,x'\}$ with~$N_G(u)\cap V(H) \setminus \{x,x'\} = N_G(x) \cap V(H)\setminus \{x,x'\}$. Note that in particular~$\{x,x'\} \subseteq M$. Moreover~$M\neq V(G)$ since~$H$ has a least three vertices. 
Let~$M'$ be the vertices of~$M$ that are in the connected component of~$G[M]$ containing~$x$ and~$x'$. By construction~$M'$ is a homogeneous set of~$G[M'\cup V(H)]$.
Since~$G$ is prime,~$M'$ cannot be a homogeneous set in~$G$. Recalling that~$M\neq V(G)$ we observe that there exists a vertex~$v\in V(G)\setminus M'$ that is adjacent to some vertex~$y\in M'$ and non-adjacent to some vertex~$y' \in M'$. Since~$G[M']$ is connected, we can choose~$y$ and~$y'$ to be adjacent. Note that~$v\notin M$, since~$M'$ forms a connected component of~$G[M]$. The graph~$H'$ induced by~$(V(H)\setminus \{x,x'\}) \cup \{y,y'\}$ is isomorphic to~$H$. An isomorphism~$\varphi$ can be constructed by sending~$x$ to~$y \coloneqq \varphi(x)$ and~$x'$ to~$y' \coloneqq \varphi(x')$ as well as fixing all other vertices. The graph induced by the set~$(V(H') \cup \{v\}$ 
has the property that~$\{y,y',v\}$ is not a homogeneous set since~$y\in M$ but~$v\notin M$. 
\end{proof}

We remark that we cannot necessarily guarantee that the graph~$H'$ claimed to exist by the lemma is equal to the graph~$H$. However, in applications of the lemma we are usually only interested in the existence of certain subgraphs. In this case, we can just assume that~$H$ and~$H'$ are equal and that there exists a vertex~$v\in V(G) \setminus V(H)$ adjacent to exactly one of~$x$ and~$x'$ with~$N_{G}(v)\cap V(H) \setminus\{x,x'\} \neq N_H(x) \setminus \{x'\} = N_H(x') \setminus \{x\}$.

As an example of how to apply the lemma, we give a very concise proof of a well known result of Ho{\`{a}}ng and Reed, which we require later. For this recall that the \emph{house} is the complement of a~$P_5$, that the \emph{domino} is the graph obtained from a~$C_6$ by adding an edge between two vertices of maximum distance and that the \emph{letter-$A$-graph}, which is shown in Figure~\ref{fig:the:A:Graph}, is obtained from a domino by deleting an edge whose endpoints both have degree 2.

\begin{lemma}[\cite{DBLP:journals/jgt/HoangR89}]\label{lem:4cycle:house:domnino}
A prime graph that contains a~$4$-cycle contains a house or a domino or a letter-$A$-graph.
\end{lemma}

\begin{proof}
Let~$G$ be a prime graph. Let~$x_1,x_2,x_3,x_4$ be vertices forming an induced~$4$-cycle~$C$ with~$x_1$ and~$x_3$ non-adjacent.
By Lemma~\ref{lem:the:extension:lemma} we can assume that there is a vertex~$v_1$ not in~$C$ adjacent to~$x_1$ but not adjacent to~$x_3$ with
$N_G(v_1)\cap V(C) \setminus \{ x_1,x_3\} = N_G(v_1)\cap \{ x_2,x_4\} \neq \{x_2,x_4\}$.
The vertices~$v_1,x_1,x_2,x_3,x_4$ now induce a house or a~$4$-cycle with an attached leaf. In the former case, we are done, so we assume the latter.
We apply Lemma~\ref{lem:the:extension:lemma} again and can therefore assume that there is a~$v_2 \notin (V(C)\cup\{v_1\})$ adjacent to~$x_2$ but not to~$x_4$ such that~$N_G(v_2) \cap \{v_1,x_1,x_3\}\neq \{x_1,x_3\}$.
We analyze the neighbors of~$v_2$ in~$\{x_1,x_3,v_1\}$.
If~$v_2$ is not adjacent to a vertex in~$\{x_1,x_3,v_1\}$, we get a letter-$A$-graph.
If~$v_2$ is adjacent to exactly~$v_1$, then~$\{x_1,x_2,x_3,x_4, v_1,v_2\}$ induces a domino. 
If~$v_2$ is adjacent to exactly one vertex in~$\{x_1, x_3\}$, then~$\{x_1,x_2,x_3,x_4, v_2\}$ induces a house (irrespective of the adjacency to~$v_1$). By the condition for~$N_G(v_2)$ it cannot be the case that~$v_2$ is a neighbor of~$x_1$ and~$x_3$ but not of~$v_1$.
It remains the case that~$v_2$ is a neighbor of all vertices in~$\{x_1,x_3,v_1\}$. In this case~$\{x_1,x_3,x_4, v_1,v_2\}$ induces a house.
\end{proof}

\section{Minimal involution-free graphs}

Ne\v{s}et\v{r}il and Sabidussi~\cite{DBLP:journals/gc/NesetrilS92} identified 18 minimal involution-free graphs. These are depicted in Figure~\ref{fig:min_asym_graphs} and are all asymmetric. We let~$\mathcal{C}$ 
be the class of minimal involution-free graphs that are not isomorphic to one of these 18 graphs. Our goal throughout the rest of the paper is to show that~$\mathcal{C}$ is empty.
While Figure~\ref{fig:min_asym_graphs} contains graphs that have~$P_5$ and~co-$P_5$ as induced subgraphs, graphs in~$\mathcal{C}$ are~$P_5$-free and~co-$P_5$-free~\cite{DBLP:journals/gc/NesetrilS92}.

In~\cite{DBLP:journals/dm/Fouquet93}, Fouquet shows that every prime~$(P_5,\text{co-}P_5)$-free graph is~$C_5$-free, or is isomorphic to~$C_5$. 
Recalling that all minimal involution-free graphs are prime and noting that~$C_5$ is not involution-free, we conclude that all graphs in~$\mathcal{C}$ are~$C_5$-free, in addition to being~$(P_5,\text{co-}P_5)$-free.

\begin{fact}
Graphs in~$\mathcal{C}$ are~$(P_5,\text{co-}P_5,C_5)$-free. 
\end{fact}

Recall that a subset~$M$ of the vertices of a graph~$G$ is \emph{dominating} if every vertex in~$V(G)\setminus M$ has a neighbor in~$M$. 
We use the following well known fact that was independently proven by Bacs{\'o} and Tuza~\cite{MR1113210} as well as Cozzens and Kelleher~\cite{DBLP:journals/dm/CozzensK90}.

\begin{theorem}[\cite{MR1113210,DBLP:journals/dm/CozzensK90}]\label{thm:dominating_clique}
Every finite connected graph that is~$(P_5,C_5)$-free has a dominating clique. 
\end{theorem}

Hence we conclude that every possible graph in~$\mathcal{C}$ has a dominating clique. With this observation we can rule out that there exist bipartite or co-bipartite minimal involution-free graphs in~$\mathcal{C}$ as follows.

\begin{lemma}\label{lem:bipartite_implies_P5_in_C}
There is no bipartite and no co-bipartite graph in~$\mathcal{C}$. 
\end{lemma}

\begin{proof}
Since~$\mathcal{C}$ is closed under complements, it suffices to show that there is no bipartite graph in~$\mathcal{C}$. Thus, let~$G$ be a minimal involution-free bipartite graph from~$\mathcal{C}$ on vertex set~$V$.
Then~$G$ is prime and hence connected. Thus, the bipartition~$(A,B)$ is canonical (i.e., isomorphism invariant) up to interchanging~$A$ and~$B$. 
By Theorem~\ref{thm:dominating_clique},~$G$ has a dominating clique~$C$.
The clique cannot have size one, since the graph would have a universal vertex. Thus, the dominating clique~$C$ consists of exactly two vertices~$v_1\in A$ and~$v_2\in B$, say. Since~$G$ is connected each of the two vertices in~$C$ is joined to all vertices of the other bipartition class. Moreover, each bipartition class has at least 2 vertices.

Consider first the graph~$G[V\setminus \{v_1\}]$, obtained from~$G$ by removing~$v_1$. Within this subgraph let~$M$ be the connected component containing~$v_2$. Note that~$M$ contains at least~$v_2$ and all vertices of~$A$ except~$v_1$. 
We conclude with Theorem~\ref{thm:dominating_clique} that~$G[M]$ must have a dominating clique. 
This implies that there is a vertex~$x_1\in A\setminus\{v_1\}$ adjacent to every vertex of~$M\cap B$.
If it were the case that~$M\cap B = B$, then~$x_1$ would be adjacent to all vertices of~$B$. However, in this case~$x_1$ and~$v_1$ would be true twins, which cannot be since twins from a homogeneous set, but~$G$ is prime.

Thus, there is a vertex~$u_2\in B\setminus M$. 
Since~$M\cap A = A\setminus\{v_1\}$,~$u_2$ has exactly one neighbor, namely~$v_1$. Furthermore,~$u_2$ is the only vertex in~$B\setminus M$ since all such vertices are twins.

By symmetry we can repeat the entire argument finding vertices~$u_1\in A$ and~$x_2\in B$ such that the neighborhood of~$u_1$ is exactly~$\{v_2\}$ and~$x_2$ is adjacent to all vertices in~$A$ except~$v_1$.

We conclude by considering the graph~$H \coloneqq G[V\setminus \{u_1,u_2,v_1,v_2\}]$. With~$x_1$ and~$x_2$, the graph~$H$ has at least two vertices and is connected since~$x_1$ and~$x_2$ are each connected to all vertices of the other bipartition class. Thus~$H$ has an involution~$\psi$. This involution either fixes the bipartition classes of~$H$ (as sets) or swaps them. If~$\psi$ fixes the bipartition classes, we can extend~$\psi$ to an involution of~$G$ by fixing all vertices in~$\{u_1,u_2,v_1,v_2\}$. If~$\psi$ interchanges the bipartition classes of~$H$, we can extend~$\psi$ to an involution of~$G$ by swapping~$v_1$ and~$v_2$ as well as swapping~$u_1$ and~$u_2$.
\end{proof}

Next, we rule out the possibility that graphs in~$\mathcal{C}$ are split. Recall that a \emph{split graph} is a graph whose vertex set can be partitioned into two sets~$A$ and~$B$ such that~$A$ induces a clique in~$G$ and~$B$ induces an edgeless graph in~$G$, see Figure~\ref{fig:splt:proof}. 

\begin{lemma}\label{lem:not_split_in_C}
There is no split graph in~$\mathcal{C}$.
\end{lemma}

\begin{proof}
Let~$(A,B)$ be a split partition of~$G$ with~$A$ a clique and~$B$ an independent set. Since every~$P_4$-free graph on more than one vertex is disconnected or its complement is disconnected~\cite{MR619603},~$P_4$-free graphs on at least 3 vertices are not prime. We thus conclude that~$G$ has~4 vertices~$b_1,a_1,a_2,b_2$ that form an induced path~$P_4$ in that order. This implies~$a_1,a_2\in A$ and~$b_1,b_2\in B$.

Since~$G$ is involution-free, there must be a vertex~$a_3\in A$ adjacent to exactly one vertex in~$\{b_1, b_2\}$ or there must be a vertex~$b_3\in B$ adjacent to exactly one vertex in~$\{a_1, a_2\}$. By possibly considering the complement~$\text{co-}G$ and swapping~$A$ and~$B$, we can assume that the former case happens. By symmetry, we can maintain that the degree of~$b_2$ is at least as large as that of~$b_1$. We can thus assume that~$a_3$ is adjacent to~$b_2$ but not to~$b_1$, see Figure~\ref{fig:splt:proof}. 

\begin{figure}[tbh]
\centering
\begin{tikzpicture}
	\begin{pgfonlayer}{nodelayer}
		\node [style=rn] (0) at (-1,2) {$b_1$};
		\node [style=rn] (1) at (-3, 2) {$a_1$};
		\node [style=rn] (2) at (-3, 1) {$a_2$};
		\node [style=rn] (3) at (-1,1) {$b_2$};
		\node [style=rn] (4) at (-3, -0) {$a_3$};
		%\node [style=rn] (5) at (-1,0) {};
	
		\node at (-3,3) {A};
		\node at (-1,3) {B};
	\end{pgfonlayer}
	\begin{pgfonlayer}{edgelayer}
		\draw [style=standardedge] (0) to (1);
		\draw [style=standardedge] (1) to (2);
		\draw [style=standardedge] (2) to (3);
		\draw [style=standardedge] (2) to (4);
		\draw [style=standardedge] (4) to (3);
        \draw[style=standardedge]  plot[smooth, tension=.7] coordinates {(1) (-3.5,1) (4)};
    \end{pgfonlayer}
\draw (-3,2);
\draw[dashed]  (3) ellipse (0.8 and 2.5);
\draw[dashed]  (2) ellipse (0.8 and 2.5);
\end{tikzpicture}
\caption{The figure details the proof of Lemma~\ref{lem:not_split_in_C}.}\label{fig:splt:proof}
\end{figure}
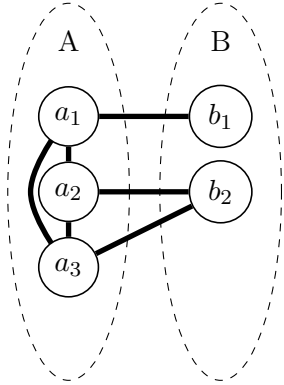

Note that in the graph induced by~$\{a_1,a_2,a_3,b_1,b_2\}$, the vertices~$a_2$ and~$a_3$ are false twins. By Lemma~\ref{lem:the:extension:lemma} we can assume that there is a vertex~$v$ adjacent to~$a_3$ and non-adjacent to~$a_2$. Recall that the application of Lemma~\ref{lem:the:extension:lemma} only provides an extension of a graph~$H'$ that is isomorphic to~$H \coloneqq G[\{a_1,a_2,a_3,b_1,b_2\}]$. Hence, in principle the vertices of~$H'$ could be distributed differently across the split partition of~$G$. However, there is only one split partition of~$H$ and we can assume~$H' = H$.

Thus, let~$v$ be adjacent to~$a_3$ and non-adjacent to~$a_2$. Since~$v$ is not adjacent to~$a_2$, it must be in~$B$. There are thus two options for~$N\coloneqq N_G(v)\cap \{a_1,a_2,a_3,b_1,b_2\}$. Either~$N = \{a_1,a_3\}$, in which case the graph induced by~$\{a_1,a_2,a_3,v,b_1,b_2\}$ is isomorphic to~$X_5$, or~$N = \{a_3\}$, in which case the graph induced by these vertices is isomorphic to~$X_4$. Since~$G$ is~$\mathcal{C}$ and thus~$(X_4,X_5)$-free the lemma follows.
\end{proof}

It is well known that a graph is a split graph if and only if it is a~$\{C_4,C_5,2K_2\}$-free graph~\cite{MR0505860}, see also~\cite{MR2063679}. Since all graphs in~$\mathcal{C}$ are~$C_5$-free, the lemma implies that all graphs in~$\mathcal{C}$ must contain a~$C_4$ or a~$2K_2$. Noting that a~$2K_2$ is the complement of a~$C_4$, it thus suffices to rule out graphs in~$\mathcal{C}$ that contain a~$C_4$, in order to show that~$\mathcal{C}$ is empty.

Recall that Lemma~\ref{lem:4cycle:house:domnino} says that prime graphs containing a~$C_4$ contain a house a domino or a letter-$A$-graph. Noting that the house is the complement of~$P_5$ and that the domino contains a~$P_5$, we conclude using Lemma~\ref{lem:4cycle:house:domnino} that graphs in~$\mathcal{C}$ or their complements must contain a letter-$A$-graph. We now investigate possible 1-vertex extensions of the letter-$A$-graph.
As the following lemma shows, either the new vertex is homogeneously connected or a true or false twin copying an existing vertex. 

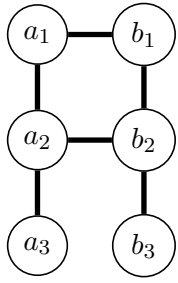
\begin{figure}[tbh]
\centering
\begin{tikzpicture}[scale = 0.70]
	\begin{pgfonlayer}{nodelayer}
		\node [style=rn] (a1) at (-3,2) {$a_1$};
		\node [style=rn] (b1) at (-1,2) {$b_1$};
		\node [style=rn] (a2) at (-3,0) {$a_2$};
		\node [style=rn] (b2) at (-1,0) {$b_2$};
		\node [style=rn] (a3) at (-3,-2) {$a_3$};
		\node [style=rn] (b3) at (-1,-2) {$b_3$};
		%\node [style=rn] (5) at (-1,0) {};	
		
	\end{pgfonlayer}
	\begin{pgfonlayer}{edgelayer}
		\draw [style=standardedge] (a1) to (b1);
		\draw [style=standardedge] (a2) to (b2);
		\draw [style=standardedge] (a1) to (a2);
		\draw [style=standardedge] (a2) to (a3);
		\draw [style=standardedge] (b1) to (b2);
		\draw [style=standardedge] (b2) to (b3);
	\end{pgfonlayer}
\end{tikzpicture}
\caption{The naming of the vertices in the letter-$A$-graph used throughout the proofs.}
\label{fig:the:A:Graph}
\end{figure}

\begin{lemma}\label{lem:extensions:of:A}
Let~$G$ be a graph in~$\mathcal{C}$. Let~$H$ be an induced subgraph of~$G$ isomorphic to the letter-$A$-graph. If~$x$ is a vertex of~$G$, then~$N_G(x)\cap V(H)$ is either empty, all of~$V(H)$, equal to~$N_H(v)$ or equal to~$\overline{N}_H(u) \coloneqq N_H(u)\cup \{u\}$ for some~$u\in V(H)$.
\end{lemma}

\begin{proof}
We differentiate cases according to the number of neighbors that~$v$ has in~$H$. 
For each of the 
options we show that if~$v$ has a certain neighborhood towards~$H$, it is either a closed or an open neighborhood of some vertex in~$H$ or it yields a graph that cannot be contained in a graph from~$\mathcal{C}$. We assume that the vertices of the graph~$H$ are named as indicated in Figure~\ref{fig:the:A:Graph}. 

\emph{Case 0:~$|N_G(v)\cap V(H)| = 0$}. Then~$N_G(v)\cap V(H) = \{\}$.

\emph{Case 1:~$|N_G(v)\cap V(H)| = 1$}. By symmetry we can assume that~$v$ is adjacent to~$a_1$,~$a_2$, or~$a_3$. If~$v$ is adjacent to~$a_1$, then~$\{v,a_1,a_2,b_2,b_3\}$ (and~$\{v,a_1,b_1,b_2,b_3\}$) induces a~$P_5$. If~$v$ is adjacent to~$a_2$, then~$v$ and~$a_3$ have the same open neighborhood in~$H$. If~$v$ is adjacent to~$a_3$, the set~$\{v,a_3,a_2,b_2,b_3\}$ (and~$\{v,a_3,a_2,a_1,b_1\}$) induces a~$P_5$.

\emph{Case 2:~$|N_G(v)\cap V(H)| = 2$}. 
If~$v$ has~$a_1$ as a neighbor, then to avoid a~$P_5$, vertex~$v$ must have a neighbor among~$\{a_2,b_2,b_3\}$ and among~$\{b_1,b_2,b_3\}$. If~$v$ is adjacent to~$a_1$ and~$b_2$, then~$v$ and~$b_1$ have the same open neighborhood in~$H$.
If~$v$ is adjacent to~$a_1$ and~$b_3$, then~$\{v,a_1,b_1,b_2,b_3\}$ induces a~$C_5$. 

If~$v$ has~$a_3$ as a neighbor, then to avoid a~$P_5$, vertex~$v$ must have a neighbor among~$\{a_2,b_2,b_3\}$ and among~$\{a_2,a_1,b_1\}$. However, if~$v$ is adjacent to~$a_3$ and~$a_2$, then the closed neighborhoods of~$v$ and~$a_3$ are the same when restricted to~$V(H)$. 
%If~$v$ is adjacent to~$a_3$ and~$b_1$, then~$\{v,a_3,a_2,a_1,b_1\}$ induces a~$C_5$. 

By symmetry, the only remaining option is that the neighborhood of~$v$ in~$H$ is~$\{a_2,b_2\}$, in which case~$\{v,a_2,a_1,b_1,b_2\}$ induces a house.

\emph{Case 3:~$|N_G(v)\cap V(H)| = 3$}.
Suppose that the vertex~$v$ is adjacent to exactly one vertex of the~$4$-cycle~$a_1,b_1,b_2,a_2$. Up to symmetry we then have the two options~$\{a_1,a_3,b_3\}$ or~$\{a_2,a_3,b_3\}$ for the neighborhood of~$v$. In the former case~$\{v,a_3,a_2,b_2,b_3\}$ induces a~$C_5$ and in the latter case~$\{a_3,v,b_3,b_2,b_1\}$ induces a~$P_5$.

Suppose~$v$ is adjacent to exactly two vertices of the~$4$-cycle~$a_1,b_1,b_2,a_2$ and suppose that the two neighbors of~$v$ are adjacent. Then~$\{v,a_1,a_2,b_1,b_2\}$ always induces a house. Now suppose that the two neighbors of~$v$ are not adjacent. Up to symmetry we then have that the neighborhood of~$v$ is either~$\{a_1,b_2,b_3\}$ or~$\{a_1,a_3,b_2\}$. In the former case~$\{v,a_1,b_1,b_2,b_3\}$ induces a house in the latter case~$v$ has the same neighborhood as~$a_2$.

It remains the case that~$v$ is adjacent to exactly three vertices of the~$4$-cycle~$a_1,b_1,b_2,a_2$. Up to symmetry we get as neighborhood of~$v$ the options~$\{a_1,a_2,b_1\}$ and~$\{a_1,a_2,b_2\}$. The former case is the closed neighborhood of~$a_1$, in the latter case~$\{v,a_1,a_2,a_3,b_2,b_3\}$ induces the graph~$X_4$, see Figure~\ref{fig:min_asym_graphs}.

\emph{Case 4:~$|N_G(v)\cap V(H)| = 4$}. As in Case~$3$, if~$v$ is adjacent to exactly two adjacent vertices on the~$4$-cycle~$a_1,b_1,b_2,a_2$, then~$\{v,a_1,b_1,a_2,b_2\}$ induces a house. 
Suppose~$v$ is not adjacent to~$a_1$. By the previous argument, neither~$a_2$ nor~$b_1$ can be a non-neighbor. If the other non-neighbor is~$a_3$, then the closed neighborhood of~$v$ and~$b_2$ are the same when restricted~$H$. If the other non-neighbor is~$b_2$ or~$b_3$, then~$\{v,a_1,a_2,a_3,b_1\}$ induces a house. Thus~$a_1$ is a neighbor of~$v$. By symmetry~$b_1$ is also a neighbor of~$v$. 

Up to symmetry the only possible options that remain for the non-neighborhood of~$v$ in~$V(H)$ are~$\{a_2,a_3\}$,~$\{a_2,b_3\}$, and~$\{a_3,b_3\}$. In the first case~$\{v,a_1,a_2,b_1,b_2,b_3\}$ induces the graph~$X_8$. In the second case~$\{v,a_1,a_2,a_3,b_1\}$ induces a house and in the third case~$\{v,a_2,a_3,b_1,b_2,b_3\}$ induces the graph~$X_4$, see Figure~\ref{fig:min_asym_graphs}.

\emph{Case 5:~$|N_G(v)\cap V(H)| = 5$}. By symmetry we can assume that~$v$ has a non-neighbor in~$\{a_1,a_2,a_3\}$. If~$v$ is not adjacent to~$a_1$, then~$\{v,a_1,a_2,a_3,b_1,b_2\}$ induces the graph~$X_8$, see Figure~\ref{fig:min_asym_graphs}. If~$v$ is not adjacent to~$a_2$, then~$\{v,a_1,a_2,b_1,b_2,b_3\}$ induces the graph~$X_8$. If~$v$ is not adjacent to~$a_3$, then~$\{v,a_1,a_2,a_3,b_2,b_3\}$ induces the graph~$X_5$.

 \emph{Case 6:~$|N_G(v)\cap V(H)| = 6$}. Then~$N_G(v)\cap V(H) = V(H)$.
\end{proof}

By combining Lemmas~\ref{lem:the:extension:lemma} and~\ref{lem:extensions:of:A} it is possible to rule out that the new vertex is a false twin and rule out that the new vertex is adjacent to all vertices. This strengthens Lemma~\ref{lem:extensions:of:A} as follows.

\begin{lemma}\label{lem:extensions:of:A:improved}
Let~$G$ be a graph in~$\mathcal{C}$. Let~$H$ be an induced a subgraph of~$G$ isomorphic to the letter-$A$-graph. If~$v$ is a vertex of~$G$, then~$N_G(v)\cap V(H)$ is either empty or equal to~$N_H(u)$ for some~$u\in V(H)$.
\end{lemma}

\begin{proof}
Let us first list the possible neighborhoods~$N_G(v)\cap V(H)$ that a vertex~$v\in V(H)$ can have towards the letter-$A$-graph~$H$ according to Lemma~\ref{lem:extensions:of:A}. These are the six open and six closed neighborhoods
\begin{center}
$
\begin{array}{l@{}c@{}l l@{}c@{}l}
%\{\} & V(H) = \{a_1,a_2,a_3,b_1,b_2,b_3\}\\
N_H(a_1) &=& \{a_2,b_1\}&
\overline{N}_H(a_1)&=& \{a_1,a_2,b_1\}\\ 
N_H(b_1) &=& \{b_2,a_1\}&
\overline{N}_H(b_1) &=& \{b_1,b_2,a_1\}\\
N_H(a_2) &=& \{a_1,a_3,b_2\}&
\overline{N}_H(a_2) &=& \{a_1,a_2,a_3,b_2\}\\
N_H(b_2) &=& \{b_1,b_3,a_2\}&
\overline{N}_H(b_2) &=& \{b_1,b_2,b_3,a_2\}\\
N_H(a_3) &=& \{a_2\}& 
\overline{N}_H(a_3) &=& \{a_2,a_3\}\\
N_H(b_3) &=& \{b_2\}&
\overline{N}_H(b_3) &=& \{b_2,b_3\},
\end{array}
$
\end{center}
as well as the two homogeneous sets~$\{\}$ and~$V(H) = \{a_1,a_2,a_3,b_1,b_2,b_3\}$.
For the moment, let us call these 14 sets \emph{(allegedly) admissible}.  

\emph{Case 1:~$N_G(v)\cap V(H) = V(H)$}. 
Suppose that~$H$ is an induced copy of the~letter-$A$-graph and~$v$ is adjacent to all vertices of~$H$. Let~$M$ be a set obtained by starting with~$V(H)$ and repeatedly adding one vertex
that has both a neighbor and a non-neighbor in the set created so far, until no such vertex remains. Then~$M$ is a homogeneous set containing~$V(H)$ by construction. However, we can argue that~$v$ is not in~$M$ as follows.

Let~$M_1 \subseteq M_2 \subseteq \ldots \subseteq M_\ell = M\subseteq V(G)$ be a sequence of vertex sets such that~$M_1 = V(H)$,~$M_{i+1}\setminus M_{i} = \{y_{i+1}\}$ and such that~$y_{i+1}$ has a neighbor and a non-neighbor in~$M_i$. 
We claim that~$M$ has the property that for all~$y\in M$ there is a sequence~$( N_1^y ,N_2^y,\ldots,N_t^y)$  
of subsets of~$M$ each of size~$6$ such that each~$G[N^y_i]$ is isomorphic to the letter-$A$-graph,~$N^y_1=V(H)$,~$y\in N^y_t$, and~$N^y_{i+1}$ is obtained from~$N^y_i$ by replacing exactly one vertex. 
We show this claim by induction on~$\ell$. 

If~$\ell = 1$ then~$y\in V(H)$ and the sequence~$N^y_1 = V(H)$ already satisfies the requirements. 
Let~$j$ be the smallest integer such that~$y$ is not uniformly connected (i.e., completely connected or completely disconnected) to~$M_j$
and let~$(N^{y_j}_1,N^{y_j}_2,\ldots,N^{y_j}_t)$ be the corresponding sequence of letter-$A$-graphs for~$y_j$ that exists by induction. Then~$y$ 
is not uniformly connected to~$N^{y_j}_1 \cup \{y_j\}$. Thus there is~$k\in \{1,\ldots,t\}$ such that~$y$ is non-uniformly connected to~$N^{y_j}_k$. 
Since~$y$ is not uniformly connected, Lemma~\ref{lem:extensions:of:A} implies that there exists a vertex~$w \in N^{y_j}_k$ that has in $N^{y_j}_k$ the same neighborhood as~$y$. The sequence~$N^{y_j}_1,\ldots,N^{y_j}_k,N_{k+1} \coloneqq (N^{y_j}_k \cup \{y\})\setminus \{w\}$ satisfies the required properties. 

We now argue that~$v\notin M$. Indeed, it suffices to show that~$v$ is adjacent to all vertices of~$M$. Suppose~$v$ is not adjacent to~$y\in M$ and let~$(N^{y}_1= V(H),\ldots,N^{y}_t)$ be the sequence of sets for~$y$ described above. Since~$v$ is adjacent to all six vertices of~$N^{y}_1$, there is an~$N^{y}_i$ such that~$v$ is adjacent to exactly five vertices of~$N^{y}_i$ which contradicts~Lemma~\ref{lem:extensions:of:A}. This proves Case 1.
\medskip

In the remaining cases, we argue that~$v$ cannot be a false twin. We consider the cases that~$N_G(v) = \overline{N}_H(a_i)$ for some~$i\in \{1,2,3\}$. The cases that~$N_G(v) = \overline{N}_H(b_i)$ then follow by symmetry.

\emph{Case 2:~$N_G(v)\cap V(H) = \{a_1,a_2,b_1\} = \overline{N}_H(a_1)$}. By applying Lemma~\ref{lem:the:extension:lemma} to the graph~$G[V(H)\cup \{v\}]$ we can assume that in~$G$ there is a vertex~$y$ adjacent to~$v$ but not adjacent to~$a_1$ such that the neighborhood of~$y$ in~$V(H)$ is not~$\{a_2,b_1\}$.
The neighborhood of~$y$ in~$V(H)$, i.e.,~$N \coloneqq N_G(y)\cap V(H)$, must be among the admissible extensions. Consider the graph~$G[V(H)\setminus\{a_1\}\cup \{v\}]$ obtained from~$H$ by replacing~$a_1$ with~$v$. The neighborhood of~$y$ in this graph is~$N \cup\{v\}$. The graph~$G[V(H)\setminus\{a_1\}\cup \{v,y\}]$ is a 1-vertex extension of a letter-$A$-graph. Thus the neighborhood of~$y$ in this graph must be obtained from one of the admissible sets by replacing~$a_1$ with~$v$. This implies that~$N' \coloneqq N\cup \{a_1\}$ must be admissible. Inspecting the options for~$N$ and~$N'$ we conclude the only possible choice is~$N = \{b_2\}$ and~$N' = \{a_1,b_2\}$. However, in that case the set~$\{y,a_1,b_1,b_2,v\}$ induces a house.

\emph{Case 3:~$N_G(v)\cap V(H) = \{a_1,a_2,a_3,b_2\} = \overline{N}_H(a_2)$}. In analogy to before, by Lemma~\ref{lem:the:extension:lemma}, we can assume that in~$G$ there is a vertex~$y$ adjacent to~$v$ but not adjacent to~$a_2$ such that~$(N_G(y)\cap V(H)) \setminus \{a_2\}\neq \{a_1,a_3,b_2\}$. We conclude that~$N\coloneqq N_G(y)\cap V(H)$ and~$N' \coloneqq N\cup \{a_2\}$ must be admissible. In this case, the only possible choice is that~$N = \{\}$ and~$N' = \{a_2\}$.
However, in that case the set~$\{y,a_1,a_2,b_2,b_3,v\}$ induces~$X_4$, see Figure~\ref{fig:min_asym_graphs}.

\emph{Case 4:~$N_G(v)\cap V(H) = \{a_3,a_2\} = \overline{N}_H(a_3)$}. Again, by Lemma~\ref{lem:the:extension:lemma}, we can assume that in~$G$ there is a vertex~$y$ adjacent to~$v$ but not adjacent to~$a_3$ such that~$(N_G(y)\cap V(H)) \setminus \{a_3\}\neq \{a_2\}$. We again conclude that~$N\coloneqq N_G(y)\cap V(H)$ and~$N' \coloneqq N\cup \{a_3\}$ must be admissible. 
Then~$N = \{a_1,b_2\}$ and~$N' = \{a_1,b_2,a_3\}$ and in this case the set~$\{a_3,v,y,b_2,b_3\}$ induces a~$P_5$.
\end{proof}

The proof of the lemma thus shows that out of the~$14$ sets that were allegedly admissible only~$7$ (the open neighborhoods and the empty set) remain in question.
Inspecting these~$7$ sets, and observing that the letter-$A$-graph is bipartite, the previous lemma shows that every 1-vertex extension of that graph is also bipartite. Furthermore, every connected 1-vertex extension has a pair of true twins.
Using this we can show that graphs in~$\mathcal{C}$ are bipartite, co-bipartite or split as follows. 

\begin{lemma}
\label{lem:min_asym_A_bipartite}\label{lem:the_smoking_lemma}
Every graph in~$\mathcal{C}$ is bipartite, co-bipartite or split.
\end{lemma}

\begin{proof} 
Let~$G$ be a graph in~$\mathcal{C}$ that is not split. Since~$G$ is~$C_5$-free and not split,~$G$ or its complement contain a~$C_4$~\cite{MR0505860}. We may assume the former. Since~$G$ is prime, Lemma~\ref{lem:4cycle:house:domnino} implies that~$G$ contains an induced subgraph~$H$ isomorphic to the letter-$A$-graph.
Since~$G$ does not contain a~$C_5$ or a~$P_5$, it does not contain odd induced cycles of length at least~5. To show that~$G$ is bipartite it thus suffices to show that~$G$ is triangle-free.

Suppose~$G$ contains a triangle~$T$.
Since~$H$ is triangle-free, not all vertices of~$T$ can be vertices of~$H$. Moreover, by Lemma~\ref{lem:extensions:of:A:improved} no 1-vertex extension of~$H$ has a triangle. 
We conclude that~$T$ and~$H$ intersect in at most one vertex. 

We now choose isomorphic copies of~$H$ and~$T$ to be as close as possible (i.e., such that the shortest distance from a vertex in~$H$ to a vertex in~$T$ is minimal). Then, by induction we see that~$V(H)$ and~$V(T)$ intersect as follows. Assuming they do not intersect, let~$h\in V(H)$ and~$t\in V(T)$ be vertices that minimize the distance between~$H$ and~$T$. Let~$y$ be the neighbor of~$h$ on a shortest path from~$h$ to~$t$. By Lemma~\ref{lem:extensions:of:A:improved}, in~$G[V(H)\cup \{y\}]$,~$y$ is a twin of some vertex~$w\in V(H)$. Then~$(V(H)\setminus \{w\}) \cup \{y\}$ induces a copy of the~letter-$A$-graph that is closer to~$T$ than~$H$. We can therefore assume that~$H$ and~$T$ intersect in~$h = t$. 

Let~$t'$ be another vertex in~$T$. By Lemma~\ref{lem:extensions:of:A:improved}, in~$G[V(H)\cup \{t'\}]$, the vertex~$t'$ has the same neighbor as some other vertex~$h'\in H$. If~$h'\neq t$, then the graph~$G[V(H)\setminus\{h'\} \cup\{t'\}]$ is an induced copy of the~letter-$A$-graph that shares two vertices with~$T$, a possibility that we already ruled out.
We conclude that~$t'$ and~$t = h'$ have the same neighborhood in~$H$. Moreover, they are adjacent. However, this contradicts Lemma~\ref{lem:extensions:of:A:improved}. \end{proof}

We have assembled the required information about~$\mathcal{C}$ to prove the main theorems.

\begin{proof}[Proof of Theorems~\ref{thm:min_asym_graphs} and~\ref{thm:min:inv:free}]
Lemmas~\ref{lem:bipartite_implies_P5_in_C} and~\ref{lem:not_split_in_C} in conjunction with Lemma~\ref{lem:min_asym_A_bipartite} show that the set~$\mathcal{C}$ of minimal involution-free graphs not depicted in Figure~\ref{fig:min_asym_graphs} is empty. This proves Theorem~\ref{thm:min:inv:free}. Since all graphs in Figure~\ref{fig:min_asym_graphs} are asymmetric, Lemma~\ref{lem:min:inv:free:impl:asym} implies that the minimal involution-free graphs are exactly the minimal asymmetric graphs. This proves Theorem~\ref{thm:min_asym_graphs}.
\end{proof}

%\section{Acknowledgments}

\bibliographystyle{abbrv}
\bibliography{main}

\end{document}